\documentclass{birkjour}

\usepackage{color}
 \newtheorem{thm}{Theorem}[section]
 \newtheorem{cor}[thm]{Corollary}
 \newtheorem{lem}[thm]{Lemma}
 \newtheorem{prop}[thm]{Proposition}
 \theoremstyle{definition}
 \newtheorem{defn}[thm]{Definition}
 \theoremstyle{remark}
 \newtheorem{rem}[thm]{Remark}
 
 \numberwithin{equation}{section}
\usepackage{amssymb}
 \usepackage{enumitem}

\def\dis{{\rm dis}}
\def\dim{{\rm dim}}
\def\codim{{\rm codim}}
\def\N{\mathcal{N}}
\def\R{\mathcal{R}}
\def\NN{\mathbb{N}}
\def\B{\mathcal{B}}

\def\CC{\mathbb{C}}
\begin{document}

%
%
%
%
%
%
%
%
%

\title[New properties  of some classes of Saphar type operators]
 {New properties  of some classes of Saphar type operators}



\author[Ayoub Ghorbel]{Ayoub Ghorbel}

\address{%
Dynamical and Combinatorial Systems LR13ES20, Faculty of Sciences of Sfax, University of Sfax, Tunisia}
\email{aghorbel39@gmail.com}

\author[Sne\v zana \v C. \v Zivkovi\'c-Zlatanovi\'c]{Sne\v zana \v C. \v Zivkovi\'c Zlatanovi\'c}

\address{%
University of Ni\v s,
Faculty of Sciences and Mathematics, Ni\v s, Serbia}

\email{snezana.zivkovic-zlatanovic@pmf.edu.rs, mladvlad@mts.rs}




\begin{abstract}
This  paper  explores additional properties of some classes of Saphar type operators, namely left Drazin invertible, essentially left Drazin invertible, right Drazin invertible, and essentially right Drazin invertible operators on Banach spaces, building upon the groundwork laid in \cite{GM} and \cite{ZS}. Specifically, we propose alternative definitions for these operators and characterize them with a new type of operator decomposition, providing a deeper understanding of their  properties. Furthermore,  we   investigate their behavior under powers.  The operators we study are distinguished from other operators bearing the same name in existing literature, see \cite{Aiena}, \cite{Q. Jiang}. By employing more refined definitions, we uncover a broader range of properties for these operators, setting them apart from their counterparts in the literature.

\end{abstract}

\maketitle

\section{Introduction}

Let $ X $   be  an infinite-dimensional complex Banach space and  let $ \mathcal{B}(X) $ be the  algebra of  all linear bounded operators on $ X$.
For $ T \in \mathcal{B}(X) $, let $\mathcal{N}(T) $ denote its kernel and  $\mathcal{R}(T) $ denote its range, and let  $a(T)$ (resp., $d(T)$, $a_e(T)$, $d_e(T)$) denote its ascent (resp., descent,  essential ascent, essential descent).
Recall  that an operator $ T \in \mathcal{B}(X) $ is called Drazin invertible if there exist a nonnegative integer $ k $  and a (unique) bounded  operator $ S $  such that
$$ ST = TS, STS = S, ST^{k+1} = T^k.$$
The smallest $ k $ for which the definition is valid is called the index of $ T $.
It is well known that Drazin invertible operators  can be characterized in different ways, for instance, we have equivalences between the  following statements:
 \begin{itemize}
       \item  $ T  $ is Drazin invertible.
    \item $ a(T) = d(T) < \infty$.
       \item  There exist   a pair $(M,N) \in Red(T)$  such that $T_{M} $ is  invertible  and  $T_{N} $ is  nilpotent.
   \end{itemize}
   The last characterization is called  the Kato decomposition of Drazin invertible operators. In addition to that, Drazin invertible operators can be also characterized in the following way

\begin{thm}{\upshape\cite{king}}\label{1000} Let $ T \in \mathcal{B}(X) $. Then, the  following two statements are equivalent:
\begin{enumerate}[label={\upshape(\roman*)}]

\item $ T $ is  Drazin invertible.
\item There exist operators $ S, R \in \mathcal{B}(X) $ such that:
\begin{enumerate}[label={\upshape(\alph*)}]
\item $ T = S + R $ and $ SR = RS = 0 $.
\item $ S $ is  Drazin invertible with index $ 0 $ or  $1$.
\item $ R $ is nilpotent.
\end{enumerate}
\end{enumerate}
\end{thm}
In \cite{GM}, the concept of Drazin invertibility is extended through the introduction of two new classes of operators: left Drazin invertible and right Drazin invertible operators, defined as follows:

\begin{defn}  {\upshape An operator  $ T \in \mathcal{B}(X) $ is called} left Drazin invertible {\upshape if $ p := a(T) < \infty $ and  $ \mathcal{R}(T) + \mathcal{N}(T^p) $ is topologically complemented, while $ T \in \mathcal{B}(X) $  is called} right Drazin invertible {\upshape if $ q := d(T) < \infty $ and  $ \mathcal{N}(T) \cap \mathcal{R}(T^q) $ is  topologically complemented.}
\end{defn}

Left and right Drazin invertible operators possess several notable properties, the most significant of which, akin to Drazin invertible operators, is the Kato decomposition property.
 \begin{thm}\label{left}{\upshape\cite{GM}}
Let $ T \in \mathcal{B}(X) $. Then $ T  $ is left (right) Drazin invertible if and only if
there exist  a pair $(M,N) \in Red(T)$   such that  $T_{M} $ is  left (right) invertible  and  $T_{N} $ is  nilpotent of degree $a(T)$ ($d(T)$).
   \end{thm}

In \cite{ZS}, broader classes of left and right Drazin invertible operators, termed essentially left Drazin invertible and essentially right Drazin invertible operators, are introduced as follows:

\begin{defn}\label{df1} \rm An operator  $T\in \mathcal{B}(X)$ is said to be  essentially left
 Drazin invertible if
 $a_e(T) < \infty$  and the subspace  $ \mathcal{R}(T)+\mathcal{N}(T^{\dis (T)})$ is topologically complemented, while
   $T\in \mathcal{B}(X)$ is said to be  essentially right
 Drazin invertible if
 $d_e(T) < \infty$ and   $ \mathcal{N}(T)\cap \mathcal{R}(T^{\dis (T)})$ is topologically complemented, where $\dis (T)$ is the  degree of stable iteration of $T$.
\end{defn}
Like left (resp., right) Drazin invertible operators, essentially left (resp., right) Drazin invertible operators are distinguished by the Kato decomposition property: they can be expressed as a direct sum of a left (resp., right) Fredholm operator and a nilpotent operator.

 \begin{thm}\label{Kato_essen}{\upshape\cite{ZS}}
Let $ T \in \mathcal{B}(X) $. Then $ T  $ is essentially  left (right) Drazin invertible if and only if
there exist  a pair $(M,N) \in Red(T)$   such that  $T_{M} $ is  left (right) Fredholm  and  $T_{N} $ is  nilpotent.
   \end{thm}
Consider now the sets
 \begin{eqnarray*}
   R_1 &\! =\! & \{T\! \in \! \B(X) :  a(T)\! <\! \infty,\!  \ \mathcal{R}(T^{n})\ {\rm is\! \ closed\! \ for\! \ every}\! \ n\! \ge\!  a(T)\},\nonumber \\
  R_2 &\! =\! & \{T\! \in\!  \B(X): d(T)\! <\! \infty,\!  \ \mathcal{R}(T^{n})\ {\rm is\! \ closed\! \ for\! \ every}\ n\ge d(T)\},\nonumber \\
  R_3 \! &=\! & \{T\! \in\!  \B(X): a_e(T)\! <\! \infty,\!  \ \mathcal{R}(T^{n})\ {\rm is\! \ closed\! \ for\! \ every}\! \ n\! \ge\!  a_e(T)\},
   \\
  R_4 &\! =\! & \{T\! \in\!  \B(X): d_e(T)\! <\! \infty,\!  \ \mathcal{R}(T^{n})\ {\rm is\! \ closed\! \ for\! \ every}\! \ n\! \ge\!  d_e(T)\}.
 \end{eqnarray*}

It is well known that a  bounded linear  operator $T$ acting on a Hilbert space   $T$ belongs to $R_1$ (resp., $ R_2$, $ R_3$, $ R_4$) if and only if  it can be decomposed into  a direct sum of a nilpotent operator and  a left invertible one  (resp.,  right invertible, left Fredholm,  right Fredholm)   \cite[Theorem 3.12]{P8}. It was shown in \cite{ZS} that this result does not hold in the general Banach space setting: the class of bounded linear  operators on a Banach space  which are characterized  by the property that they can be decomposed into a direct   sum of a nilpotent operator and   a left invertible (resp., right invertible,  left Fredholm,  right Fredholm) one, i.e. the class of left (resp., right,  essentially left,  essentially right) Drazin invertible operators,   can be  a proper subset of the class $R_1$  (resp., $ R_2$, $ R_3$, $ R_4$).
For a bounded linear operators $T$ acting on a Banach space and belonging to the class $R_1$ (resp., $R_2$, $R_3$, $R_4$) to be decomposed into a direct sum of a nilpotent operator and   a left invertible one  (resp., right invertible,  left Fredholm,  right Fredholm),  it is necessary and sufficient that $T$ is of Saphar type, i.e. that   the subspaces $ \mathcal{N}(T)\cap \mathcal{R}(T^{\dis (T)})$ and  $ \mathcal{R}(T)+\mathcal{N}(T^{\dis (T)})$ are topologically complemented, where $\dis (T)$ is the  degree of stable iteration of $T$ (see \cite[Theorems 4.13, 4.15,  Corolarries 4.23. 4.24]{ZS}, \cite[Theorems 3.2,3.7]{GM}).
Even though operators from the class  $R_1$ (resp., $ R_2$, $ R_3$, $ R_4$)  were previously known in other literature   as left Drazin invertible (resp., right Drazin inverible, essentially left Drazin invertible, essentially right Drazin invertible) operators, we refer to them in our work  as  upper Drazin invertible operators (resp., lower Drazin invertible, essentially upper Drazin invertible, essentially lower Drazin invertible).   We observed that the name left Drazin invertible  (resp., right Drazin inverible, essentially left Drazin invertible, essentially right Drazin invertible) operators is more appropriate for operators that can be  represented as a direct sum  of a nilpotent operator and   a left invertible  (resp., right invertible,  left Fredholm,  right Fredholm) operator \cite{GM}, \cite{ZS}.\\

In this paper, we seek to develop further properties of left Drazin invertible, essentially left Drazin invertible, right Drazin invertible, and essentially right Drazin invertible operators on Banach spaces. We begin our study with providing necessary and sufficient conditions for an upper (resp., lower, essentially upper, or essentially lower) Drazin invertible operator to be left (resp., right, essentially left, or essentially right) Drazin invertible, thereby distinguishing these newly introduced concepts from traditional ones.
We then present several characterizations that serve as alternative definitions for these new operators. To this end, we introduce and prove two new foundational lemmas that establish sufficient conditions for a paracomplete subspace to be topologically complemented. These lemmas are  particularly instrumental  in  establishing that  necessary conditions for   an operator  $T\in B(X)$ to be essentially left (resp., essentially right) Drazin invertible are that $\R(T^n)$ and $\N(T^n)$ are topologically complemented for each $n\ge a_e(T)$ (resp., $n\ge d_e(T)$), which is not the case with operators from the class $R_3$ (resp., $R_4$)
  on an arbitrary Banach space. 
 Subsequently, we introduce a new decomposition for Saphar-type operators. This decomposition involves splitting the operator into two parts: a part of Saphar type and a nilpotent part. This decomposition provides a unique characterization of Saphar-type operators, offering a new perspective on their structure. Finally in the last section, we prove that  an operator $T\in B(X)$ is (essentially) left (right) Drazin invertible if and only if its power $T^m$
is also (essentially) left (right) Drazin invertible for some positive integer $m$. This result  is also based on the previously mentioned  two lemmas and  highlights the stability of the (essentially) left (right) Drazin invertibility property under the operation of taking powers.
\section{Definitions and preliminary results}
Let $ X $   be  an infinite-dimensional complex Banach space. 
Here $\mathbb{N} \, (\mathbb{N}_0)$ denotes the set of all positive
(non-negative) integers, and $\CC$  denotes the set of all complex numbers.
 Recall the definitions of the ascent $ a(T) $ and  the descent $ d(T) $ of $ T $: we say that $ T $  has finite ascent if the chain $ \mathcal{N}(T^0) \subseteq \mathcal{N}(T)  \subseteq \mathcal{N}(T^2)  ...$ becomes constant after a finite number of steps.  The smallest  $ n\in \mathbb{N}_0$  such that $ \mathcal{N}(T^n) = \mathcal{N}(T^{n+1}) $  is defined to be $ a(T) $. The descent is defined similarly for the chain $  \mathcal{R}(T^0)  \supseteq \mathcal{R}(T)  \supseteq \mathcal{R}(T^2) ...$. We define the infimum of the empty set to be  $\infty$.
For $T\in \mathcal{B}(X)$ and  $n\in\NN_0$ we set
$\alpha_n(T)=\dim (\N(T)\cap \R(T^n))$
and
   $\beta_n(T)=\codim (\R(T)+\N(T^n)).$ The smallest integer $ n$  such that $\alpha_{n}(T)<\infty  $  is defined to be   the  essential ascent $ a_e(T) $ of $T$, while
   the smallest integer $ n$  such that $\beta_{n}(T)<\infty  $  is defined to be   the  essential descent $ d_e(T) $ of $T$.
   If there is $d\in\NN_0$ for which the sequence $(\N(T)\cap \R(T^n) )$ is  constant for $n\ge d$, then $T$ is said to have uniform descent for $n\ge d$.
   For $T\in \mathcal{B}(X)$ the  degree of stable iteration $\dis(T)$  is defined by:
$$
\dis(T)=\inf\{n\in\NN_0:m\ge n, m\in\NN\Longrightarrow \N(T)\cap \R(T^n)= \N(T)\cap \R(T^m)\}.
$$
An operator $T\in \B(X)$ is said to be {\it quasi-Fredholm of degree} $d$  if there exists a $d\in\NN_0$ such that
$\dis(T)=d$  and
 $\R(T^n)$ is closed for each $n\ge d$. 
An operator $T \in
\B(X)$ is   {\em Kato} if $\R(T)$ is closed and $\N(T) \subset
\R(T^n)$ for every $ n \in \mathbb{N}_0$. We recall that $T\in\B(X)$ is Kato if and only if $T$ is quasi-Fredholm of degree $0$.

An operator $T\in \B(X)$ is called upper (resp., lower) semi-Fredholm,
or $T\in\Phi_+(X)$ (resp., $T\in\Phi_-(X)$) if
 $\alpha (T):=\dim \N(T)<\infty$ and $\R(T)$  is closed (resp., $\beta (T):=\codim\R(T)<\infty$),
and it is said to be relatively regular 
if
 $\R(T)$ and $\N(T)$
 are topologically  complemented subspaces of  $X$. An operator $T\in \B(X)$ is called  left (resp., right) Fredholm, or
 $T\in\Phi_l(X)$ (resp., $T\in\Phi_r(X)$) if it is relatively regular  and upper (resp., lower)
 semi-Fredholm.
\

If $T\in\B(X)$ is a relatively regular Kato operator, then we say that $T$ is a {\it Saphar operator}. The degree of a nilpotent operator $T$ is the smallest $d\in\NN_0$ such that $T^d=0$.
Let $d\in\NN_0$. An operator $T\in \B(X)$ is said to be {\it of Saphar type  (of degree} $d$) if there exists a pair $(M,N) \in Red(T)$ such that
$T_M$ is Saphar and $T_N$ is nilpotent (of degree $d$). In that case, since $\dis(T_N)=a(T_N)=d(T_N)=d$ and $\dis(T_M)=0$, it follows that $\dis(T)=\max\{\dis(T_M),\dis(T_N)\}=d$.
We recall that $T\in\B(X)$ is of Saphar type  of degree $d$ if and only if it is  a quasi-Fredholm operator with  $\dis (T)=d$  and the subspaces $\R(T)+\N(T^d)$ and $\N(T)\cap \R(T^{d})$  are topologically  complemented \cite[Theorem 4.2]{ZS}, \cite[Theorem 4.1]{mine2}. An operator $T\in \B(X)$ is said  to admit a {\em generalized Saphar decomposition} if  there exists a pair $(M,N)\in Red(T)$ such that $T_M$ is Saphar and $T_N$ is quasinilpotent \cite{MZ}.

An operator $T\in \B(X)$
is meromorphic if its non-zero spectral points are poles of its resolvent. We shall say $T\in \B(X)$ admits a {\em generalized Saphar-meromorphic decomposition} if  there exists a pair $(M,N)\in Red(T)$ such that $T_M$ is Saphar and $T_N$ is meromorphic.

\

Let  $ F $, $ G $ and $ H $ be subspaces of $ X $. We will copiously use the following modular law: \begin{equation}\label{low}
  {\rm if }\ F \subseteq H,\ {\rm  then\ }  (F + G) \cap H = F + (G \cap H ).
  \end{equation}
%

\medskip

The following lemma, upon which we will rely, was originally demonstrated in \cite[Lemma 2.3]{GM}  and later generalized in \cite[Lemma 3.3]{ZS}.
\begin{lem}\label{mine} Let $ M $ and $ N $ be closed  in  $ X $ such that $ X = M \oplus N $.  Let $ F $ be a subspace of $ M $ and $ G $ be a subspace of $ N $. Then, $ F \oplus G $ is topologically complemented in $ X $ if and only if $ F $ is topologically complemented  in $ M $ and $ G $ is topologically complemented in $ N $.
\end{lem}
\begin{lem}\label{suma komplamentarnog i konacno-dimenzionalnog} \cite{Sv} Let $M, N$ be subspaces of $X$ such that $M$ is   topologically complemented  and $N$ is finite-dimensional. Then $M+N$ is a topologically complemented subspace of $X$.
\end{lem}
\begin{lem}\label{glavna za desne} \cite{ZS}  Let $M$ be a topologically complemented subspace of $X$ and let  $N$ be a closed subspace of $X$ such that $\codim N<\infty$. Then $M\cap N$ is topologically complemented in $X$.
\end{lem}


\begin{lem}\label{l3}{\upshape\cite{mine2}}
Let $ T \in \mathcal{B}(X)$ be such that $ \mathcal{R}(T) $ is topologically complemented. If $ M $ is  topologically complemented and $ \mathcal{N}(T) \subseteq  M $,  then $ T(M) $ is topologically complemented.
\end{lem}

From Theorem 3.3 in \cite{GM} and its proof it follows that the following characterizations of left Drazin invertible operator hold:

\begin{thm}\label{t1}  For $T\in \mathcal{B}(X) $ the following statements are equivalent:
\begin{enumerate}[label={\upshape(\roman*)}]
\item   $ T  $ is left Drazin invertible.
\item $ a(T) < \infty $ and $ \mathcal{N}(T^{n}) $ and $ \mathcal{R}(T^{n}) $ are topologically complemented for every $n\ge a(T) $.
\item $  a(T) < \infty $ and there exists $n\ge a(T)$ such that $ \mathcal{N}(T^{ n}) $ and $ \mathcal{R}(T^{ n+1}) $ is topologically complemented.
\end{enumerate}
\end{thm}

\section{New Characterizations and Decompositions for (Essentially) Left and (Essentially) Right Drazin Invertible Operators}

To deepen our understanding  of (essentially) left and right  Drazin invertibility of operators and to clearly differentiate these operators from related concepts in the literature, we present the following theorem which  provides  necessary and sufficient conditions for an upper (resp., lower, essentially upper,  essentially lower) Drazin invertible operator to be left (resp., right, essentially left,  essentially right) Drazin invertible.

\begin{thm}\label{nova} Let $T\in\B(X)$ be an upper (resp., lower, essentially upper, essentially lower) Drazin invertible operator. Then the following conditions are equivalent:

\begin{enumerate}[label={\upshape(\roman*)}]
\item   $ T  $ is  left (resp., right, essentially left, essentially right) Drazin invertible.
\item $ T$ is of Saphar type.
\item $ T $  admits a  generalized Saphar decomposition.
\item $ T $  admits a  generalized Saphar-meromorphic decomposition.
\end{enumerate}
\end{thm}
\begin{proof} (i)$\Longleftrightarrow$(ii) This equivalence follows from   \cite[Corollary 4.23]{ZS} since $a(T)<\infty$.

(ii)$\Longrightarrow$(iii)$\Longrightarrow$(iv) Obvious.

(iv)$\Longrightarrow$(ii) Suppose that $ T $  admits a  generalized Saphar-meromorphic decomposition. Then there exists a pair $(M,N)\in Red(T)$ such that $T_M$ is Saphar and $T_N$ is meromorphic. Since $T_N$ is meromorphic, it follows that $T_N-\lambda I_N$ is Drazin invertible for every $\lambda\in\CC^*$, so that in particular $T_N-\lambda I_N$ is upper Drazin invertible for every $\lambda\ne 0$. On the other hand, since $T$ is upper Drazin invertible, then  $T_N$ is also upper Drazin invertible. Therefore, $\{\lambda\in\CC:T_N-\lambda I_N\ {\rm is\ not\ upper\ Drazin\ invertible}\}=\emptyset$, which, according to \cite[Corollary 3.5]{Q. Jiang} implies that the Drazin spectrum of $T_N$ is empty. Thus $T_N$ is Drazin invertible, and hence there  exists a pair $(N_1,N_2)\in Red(T_N)$ such that $T_{N_1}$ is invertible  and $T_{N_2}$ is nilpotent. Applying \cite[Lemma 3.11]{ZS}, we obtain that $T_M\oplus T_{N_1}$ is Saphar. Consequently, $T$ is of Saphar type.

The rest of the assertions can be proved similarly.
\end{proof}

Therefore, an operator $T\in\B(X)$ is left (resp., right, essentially left, essentially right) Drazin invertible if and only if $T$ is upper (resp., lower, essentially upper, essentially lower) Drazin invertible and  admits a  generalized Saphar-meromorphic decomposition.

There exists  a  bounded linear  operator $T$ acting on a Banach space    that is bounded below (resp., surjective)  with a non-complemented range (resp., null-space) \cite[Example 1.1]{Burlando}. This operator is upper (resp., lower)  Drazin invertible, and hence essentially upper (resp., essentially lower) Drazin invertible. Also, this operator is of Kato type,  but it is    not of Saphar type (see \cite[p. 170]{ZS}), and hence according to Theorem \ref{nova} it does not admit a generalized Saphar decomposition nor a generalized Saphar-meromorphic decomposition. This example shows that  for an  upper (resp., lower, essentially upper, essentially lower) Drazin invertible operator to be  left (resp., right, essentially left, essentially right) Drazin invertible it is not enough to be of Kato type. Also this example shows that
 the condition that $T$  admits a generalized Saphar-meromorphic decomposition in the previous assertion can not be omitted.

\medskip

Having established how left (resp., right, essentially left, essentially right) Drazin invertible operators are distinguished from their counterparts in the literature, we now turn our attention to presenting new characterizations of these operators. We begin by offering characterizations that serve as alternative definitions for right (resp., essentially left, essentially right) Drazin invertible operators.

\begin{thm}\label{t2}

Let $ T \in \mathcal{B}(X) $. The following statements are equivalent:
\begin{enumerate}[label={\upshape(\roman*)}]
\item   $ T  $ is right Drazin invertible.
\item $ d(T) < \infty $ and $ \mathcal{N}(T^{n}) $ and $ \mathcal{R}(T^{n}) $ are topologically complemented for every $n\ge d(T) $.
\item $  d(T) < \infty $ and there exists $n\ge d(T)$ such that $ \mathcal{N}(T^{ n+1}) $ and $ \mathcal{R}(T^{ n}) $ are topologically complemented.
\end{enumerate}
\end{thm}
\begin{proof} (i)$\Longrightarrow$(ii)
If $ T $ is right Drazin invertible, then by Theorem \ref{left}, there exist  closed $ T$-invariant subspaces $ M $ and $ N $ such that $ X = M \oplus N $ and  $ T = T_{M} \oplus T_{N} $ where $T_{M} $ is  right invertible  and $T_{N}^q =0 $ where $ q := d(T) $. It follows that $ T^n = T_M^n \oplus 0 $ for every $n\ge q$. Hence, $ \mathcal{R}(T^{n}) = M $ which is topologically complemented and $ \mathcal{N}(T^{n}) = \mathcal{N}(T_M^{n}) \oplus N $. Since $ T_M^{n} $ is right invertible, $\mathcal{N}(T_M^{n}) $ is topologically complemented in $ M $ so that  $ \mathcal{N}(T^{n}) $ is topologically complemented in $X$ via  Lemma \ref{mine}.

(ii)$\Longrightarrow$(iii) It is obvious.

(iii)$\Longrightarrow$(i)
Assume that $  d(T) < \infty $ and that $ \mathcal{N}(T^{n+1}) $ and $ \mathcal{R}(T^{n}) $ are topologically complemented for some $n\ge d(T)$. Observe that
\begin{eqnarray*}
  T^n(\mathcal{N}(T^{n+1}))    &=& \mathcal{N}(T) \cap \mathcal{R}(T^{n})  \\ &=& \mathcal{N}(T) \cap \mathcal{R}(T^{d(T)}).
 \end{eqnarray*}
Now an application of Lemma \ref{l3} gives that $ \mathcal{N}(T) \cap \mathcal{R}(T^{d(T)}) $ is topologically complemented.
\end{proof}


 Now we shift our focus to essentially left and essentially right Drazin invertible operators. To this end, we introduce two key lemmas that provide sufficient conditions for a paracomplete subspace to be topologically complemented. These lemmas are essential not only for the results immediately ahead but also for several foundational results presented later.

\begin{lem}\label{main l}

Let $ T \in \mathcal{B}(X) $ and $ M $ be a paracomplete subspace (i.e. range of a bounded operator) of $ X $ such that $ M \cap \mathcal{N}(T) $ is topologically complemented. If $ T(M) $ is topologically complemented, then    $ M $ is also topologically complemented.
\end{lem}

\begin{proof}
Let $ N $ and $ G $  be closed subspaces of $ X $ such that $T(M) \oplus N = X  $ and $ (M \cap \mathcal{N}(T)) \oplus G = X $. We claim  that $ M $ is topologically complemented  in $ X $ by $ G \cap T^{-1}(N) $. Indeed from  $M \cap \mathcal{N}(T)\subset T^{-1}(N) $,   we have  according to \ref{low} that
\begin{eqnarray*}
  M + ( G \cap T^{-1}(N))    &=& M + (M \cap \mathcal{N}(T)) +  ( G \cap T^{-1}(N))  \\ &=&       M + (((M \cap \mathcal{N}(T))  + G) \cap T^{-1}(N)) \\ &=& M +( X \cap T^{-1}(N)) \\ &=& M + T^{-1}(N).
 \end{eqnarray*}
On the other hand, since $ T(M) \subset \mathcal{R}(T) $, it follows from \cite[Lemma 3.1 (iii)]{ZS} that $$ T^{-1}(T(M) + N) = T^{-1}(T(M)) + T^{-1}(N).   $$ But $$ T^{-1}(T(M)) = M + \mathcal{N}(T). $$ Therefore
\begin{eqnarray*}
 X &=&  T^{-1}(X) \\ &=&  T^{-1}(T(M) + N) \\ &=&  M + \mathcal{N}(T) + T^{-1}(N)  \\ &=& M + T^{-1}(N).
 \end{eqnarray*}
 Consequently,
 \begin{center}
 $   M + ( G \cap T^{-1}(N) )= X   $.
 \end{center}

 Now, if $ x \in M \cap (G  \cap T^{-1}(N)) $, then $ T(x) \in T(M) \cap N = \lbrace 0 \rbrace $ so that $ x \in \mathcal{N}(T) $. Therefore,  since $ M \cap G \cap \mathcal{N}(T) = \lbrace 0 \rbrace$, it follows that  $ x = 0 $, and whence $ X = M \oplus (G \cap T^{-1}(N)) $. Now to finish the proof, since $ G \cap T^{-1}(N) $ is closed, we only need to prove that $ M $ is closed. Towards this, we will apply some results from Chapter II in \cite{lab}. Since $ M $ and  $ \mathcal{N}(T) $ are paracomplete and since $ T^{-1}(T(M)) = M + \mathcal{N}(T) $ and $ M \cap \mathcal{N}(T) $ are closed, it follows from \cite[Proposition 2.1.1]{lab}  that $ M $ is closed too.
\end{proof}

\begin{lem}\label{main r}

Let $ T \in \mathcal{B}(X) $ and $ M $ be  paracomplete  such that $ M + \mathcal{R}(T) $ is topologically complemented. If $ T^{-1}(M) $ is topologically complemented, then   $ M $ is also topologically complemented.
\end{lem}
\begin{proof}
Let $ N $ and $ H $ be closed such that $  T^{-1}(M) \oplus N  = X $ and $  (M + \mathcal{R}(T)) \oplus H  = X $. We claim that $ M $ is topologically complemented in $ X $ by $ T(N) \oplus H$. Indeed,  first, one has
\begin{eqnarray*}
 T(T^{-1}(M ) + N) &=&  T(T^{-1}(M )) + T(N)  \\  &=& ( M \cap \mathcal{R}(T)) + T(N).
 \end{eqnarray*}
 Hence
 \begin{center}
     $ ( M \cap \mathcal{R}(T)) + T(N)  = \mathcal{R}(T)  $
 \end{center}
   Since $ M + \mathcal{R}(T) + H = X $, it follows that
\begin{eqnarray*}
 M  + (T(N) \oplus H)    &=&  M + (M \cap \mathcal{R}(T))+ T(N) + H \\ &=& M + \mathcal{R}(T) + H \\ &=& X.
 \end{eqnarray*}
  On the other hand, if $ x \in M \cap (T(N) \oplus H) $, then $ x = Tn + h $ for some $ n \in N $ and $ h \in H $. Hence, $ h = x - Tn $ so that $ h \in (M + \mathcal{R}(T)) \cap H = \lbrace 0 \rbrace $.  Therefore,    $ n \in T^{-1}(M) \cap  N = \lbrace 0 \rbrace $ which implies that $ x = 0 $. Summing up, $M \oplus (T(N) \oplus H) = X $. We shall lastly prove that $ M $ and $ T(N) \oplus H $ are closed. But this is a consequence of \cite[Proposition 2.1.1]{lab}  since   $ M $ and $ T(N) \oplus H $ are paracomplete (CH. II \cite{lab}) and  $ M + (T(N) \oplus H) $ and $ M \cap (T(N) \oplus H) $ are closed.
\end{proof}

\begin{thm}\label{t1 e dod}

Let $ T \in \mathcal{B}(X) $. The following statements are equivalent:
\begin{enumerate}[label={\upshape(\roman*)}]
\item   $ T  $ is essentially left Drazin invertible.
\item $ a_e(T) < \infty $ and $ \mathcal{N}(T^{n}) $ and $ \mathcal{R}(T^{n}) $ are topologically complemented for every $n\ge a_e(T) $.
\end{enumerate}
\end{thm}
\begin{proof} (i)$\Longrightarrow$(ii): Suppose that $ T $ is essentially  left Drazin invertible. Then $a_e(T)\le \dis(T)<\infty$. From the proof of the implication (iv)$\Longrightarrow$(v) in  \cite[Theorem 4.13]{ZS} it follows that there exist  closed $ T$-invariant subspaces $ M $ and $ N $ such that $ X = M \oplus N $ and  $ T = T_{M} \oplus T_{N} $ where $T_{M} $ is  left Fredholm and $T_{N}^d =0 $ for $ d := \dis(T) $. For every $n\ge d$ we have that $T^n=T_M^n\oplus 0_N$, and so  $\R(T^n)=\R(T_M^n)$  and $\N(T^n)=\N(T_M^n)\oplus N$. Since $T_M^n$ is left Fredholm it follows that $\R(T_M^n)$ and $\N(T_M^n)$ are topologically complemented in $M$. Applying Lemma \ref{mine}, we conclude that $ \mathcal{N}(T^{n}) $ and $ \mathcal{R}(T^{n}) $ are  topologically complemented in $X$ for every $n\ge d$.

 If $a_e(T)=d$, then we have that $ \mathcal{N}(T^{n}) $ and $ \mathcal{R}(T^{n}) $ are topologically complemented for every $n\ge a_e(T) $. Suppose that $a_e(T)<d$. Then $a_e(T)\le d-1$ and hence $\alpha_{d-1}=\dim (\N(T)\cap \R(T^{d-1})<\infty$. Consequently, the subspace $\N(T)\cap \R(T^{d-1})$ is topologically complemented. Since $\R(T^{d-1})$ is paracomplete and $T(\R(T^{d-1}))=\R(T^{d})$ is topologically complemented, from
Lemma \ref{main l} it follows that   $\R(T^{d-1})$ is topologically complemented. By repeating this method we can conclude that $\R(T^n)$ is topologically complemented for every  $n\ge a_e(T)$.
On the other side, from the equivalence (i)$\Longleftrightarrow$(iii) in  \cite[Theorem 4.11]{ZS} it follows that  that $\R(T)+\N(T^{d-1})$ is topologically complemented. As $\N(T^{d-1})$ is paracomplete and since $T^{-1}(\N(T^{d-1}))=\N(T^d)$ is topologically complemented, Lemma \ref{main r} ensures that $\N(T^{d-1})$ is topolo\-gically complemented. By repeating this method we can conclude that $\N(T^n)$ is complemented for every  $n\ge a_e(T)$.


(ii)$\Longrightarrow$(i): It follows from the equivalence (i)$\Longleftrightarrow$(x) in \cite[Theorem 4.13]{ZS}.
\end{proof}

\begin{thm}\label{t2 e dod}
Let $ T \in \mathcal{B}(X) $. The following statements are equivalent:
\begin{enumerate}[label={\upshape(\roman*)}]
\item   $ T  $ is essentially  right Drazin invertible.
\item $ d_e(T) < \infty $ and $ \mathcal{N}(T^{n}) $ and $ \mathcal{R}(T^{n}) $ are topologically complemented for every $n\ge d_e(T) $.
\item $  d_e(T) < \infty $ and there exists $n\ge d_e(T)$ such that $ \mathcal{N}(T^{ n+1}) $ and $ \mathcal{R}(T^{ n}) $ are topologically complemented.
\end{enumerate}
\end{thm}
\begin{proof} (i)$\Longrightarrow$(ii):
 Let $d=\dis(T)$. From the proof of the implication (iv)$\Longrightarrow$(v) in  \cite[Theorem 4.15]{ZS} it follows that there exist  closed $ T$-invariant subspaces $ M $ and $ N $ such that $ X = M \oplus N $ and  $ T = T_{M} \oplus T_{N} $ where $T_{M} $ is  right Fredholm and $T_{N}^d =0 $.
  Now from  Lemma \ref{mine},  similarly to the proof of the implication (i)$\Longrightarrow$(ii) in Theorem \ref{t1 e dod}, we conclude that $ \mathcal{N}(T^{n}) $ and $ \mathcal{R}(T^{n}) $ are  topologically complemented for every $n\ge d $.

  If $d_e(T)=d$, then we have that $ \mathcal{N}(T^{n}) $ and $ \mathcal{R}(T^{n}) $ are topologically complemented for every $n\ge d_e(T) $. Suppose that $d_e(T)<d$. Then $d_e(T)\le d-1$ and from the equivalence (i)$\Longleftrightarrow$(iii) in \cite[Theorem 4.12]{ZS} it follows that  the subspace $\N(T)\cap \R(T^{d-1})$ is topologically complemented. As in the proof of Theorem \ref{t1 e dod}, by using Lemma
   \ref{main l} we get  that   $\R(T^{d-1})$ is topologically complemented. By repeating this method we can conclude that $\R(T^n)$ is complemented for every  $n\ge d_e(T)$.

  Further, from \cite[Proposition 4.10]{ZS} it follows that $\R(T^{d})$ is closed and hence $T^{-(d-1)}(\R(T^d))=\R(T)+\N(T^{d-1})$ is closed. Since  $\beta_{d-1}=\codim (\R(T)+ \N(T^{d-1})<\infty$, we conclude that  the subspace $\R(T)+ \N(T^{d-1})$ is topologically complemented. As in the proof of Theorem \ref{t1 e dod},  Lemma \ref{main r} ensures that $\N(T^{d-1})$ is topolo\-gically complemented. By repeating this method we can conclude that $\N(T^n)$ is complemented for every  $n\ge d_e(T)$.

(ii)$\Longrightarrow$(iii): It is obvious.

(iii)$\Longrightarrow$(i): Suppose that $  d_e(T) < \infty $ and let  $ \mathcal{N}(T^{ n+1}) $ and $ \mathcal{R}(T^{ n}) $ be topologically complemented  for some $n\ge d_e(T)$.
 Since $ T^n(\mathcal{N}(T^{n+1})) = \mathcal{N}(T) \cap \mathcal{R}(T^{n})$ and  $ \mathcal{N}(T^{n}) \subseteq \mathcal{N}(T^{n+1}) $,
applying Lemma  \ref{l3}  we obtain that $ \mathcal{N}(T) \cap  \mathcal{R}(T^n) $ is topologically complemented. Using the equivalence (i)$\Longleftrightarrow$(iv) in \cite[Theorem 4.12]{ZS} we conclude that  $ T  $ is essentially  right Drazin invertible.
\end{proof}

In the final part of this section, we establish a distinct characterization of  Saphar type operators  by demonstrating a new decomposition that contrasts with the classical Kato decomposition.

 \begin{lem}\label{result1}
Let $T \in \mathcal{B}(X)$ and  let $ M $ and $ N $ be two subspaces of $ X$ such that $(M, N)\in Red(T)$ where $T_M$ is Kato and $ T_N $ is nilpotent.  Then there exist two operators $S, R \in \mathcal{B}(X)$ such that $T= S+R$, $SR =RS =0$,   $ R$ is nilpotent and for all $ n \geq 1 $:
    \begin{eqnarray}
   \mathcal{N}(S) \cap \mathcal{R}(S^n) &=& \mathcal{N}(T_M),\label{nov1}\\
   \mathcal{R}(S)+ \mathcal{N}(S^n) &=& \mathcal{R}(T_M) \oplus N.\label{nov2}
 \end{eqnarray}
If in particular $ T$ is of Saphar type, then $ S $ is also of Saphar type.
   \end{lem}
\begin{proof}
    Let $ S = T_M \oplus 0_N$ and  $ R = 0_M\oplus T_N$. Obviously, $ R$ is nilpotent and $\mathcal{N}(S^n)  = \mathcal{N}(T_M^n) \oplus   N $ and $ \mathcal{R}(S^n) =\mathcal{R}(T_M^n)$. Since $ T_M$ is Kato,     (\ref{nov2})  follows immediately and (\ref{nov1}) follows   via using (\ref{low}). The final assertion follows easily from the construction of $ S$.

\end{proof}

  \begin{prop}\label{stability}
Let $T \in \mathcal{B}(X)$. If there exist two operators $S, R \in \mathcal{B}(X)$ such that $T= S+R$  and $SR =RS =0$, then
 \begin{enumerate}[label={\upshape(\roman*)}]
\item $\mathcal{N}(T) \cap \mathcal{R}(S^n) = \mathcal{N}(S) \cap \mathcal{R}(S^n) $ for all $ n \in \mathbb N $.
\item $\mathcal{R}(T) + \mathcal{N}(S^n) = \mathcal{R}(S) + \mathcal{N}(S^n)$ for all $n \in \mathbb N.$

\end{enumerate}
   \end{prop}

\begin{proof}
(i)  Let $n \in \mathbb N $ and let $y \in \mathcal{N}(S) \cap \mathcal{R}(S^n).$ Then $ Sy = 0 $ and $ y = S^nx$ for some $ x \in X$ so that $ Ty = (S + R)S^nx = S^{n+1}x + RS^nx $. Since
 $RS^n= 0$, it follows  that $ Ty = Sy = 0$, i.e. $ y \in \mathcal{N}(T) $. Consequently,  $y \in \mathcal{N}(T) \cap \mathcal{R}(S^n).$ A similar line of reasoning yields the converse inclusion.

(ii)  Fix $n \in \mathbb N $ and let
$y \in \mathcal{R}(S) + \mathcal{N}(S^n).$ Then for some $z \in \mathcal{N}(S^n) $ and $x \in X$, one has $y-z =  Sx=Tx-Rx$
so that $y-z+Rx \in \mathcal{R}(T).$ On the other hand, since
 $S^nR= 0$, it follows that   $Rx \in \mathcal{N}(S^n)$ and therefore  $y \in \mathcal{R}(T) + \mathcal{N}(S^n).$ The converse inclusion can be demonstrated in a similar way.
\end{proof}

   \begin{cor}\label{result3}
Let $T \in \mathcal{B}(X)$. If there exist two operators $S, R \in \mathcal{B}(X)$ such that $T= S+R$,  $SR =RS =0$ and $ R$ is nilpotent of degree $ d $, then
 \begin{enumerate}[label={\upshape(\roman*)}]
\item $\mathcal{N}(T) \cap \mathcal{R}(T^n) = \mathcal{N}(S) \cap \mathcal{R}(S^n) $ for all $ n \geq d $.
\item $\mathcal{R}(T) + \mathcal{N}(T^n) = \mathcal{R}(S) + \mathcal{N}(S^n)$ for all $n \geq d.$

\end{enumerate}
   \end{cor}
\begin{proof}
      From the binomial theorem,  for all $ n \in \mathbb N $  we have that
$
T^n = \break (S+R)^n  =  \sum_{k=0}^n \binom nk S^k R^{n-k}   = R^n + \sum_{k=1}^{n-1} \binom nk S^k R^{n-k} + S^n.
$
In the last expression,   the middle sum is zero as $SR= RS = 0$. Consequently,   $T^n=S^n $ for all $n\geq d$, and applying Proposition  \ref{stability} we obtain the equalities  (i) and (ii).
\end{proof}
Consequently, we derive the following result as an immediate application of Corollary \ref{result3}.
\begin{cor}\label{result4}
Let $T \in \mathcal{B}(X)$. If there exist two operators $S, R \in \mathcal{B}(X)$ such that $T= S+R$,  $SR =RS =0$ and $ R$ is nilpotent of degree $ d $, then $ \dis(T) < \infty $ if and only if $ \dis(S) < \infty $. We have the same result for the ascent, the essential ascent, the descent and the essential descent of $ T$ and $ S$.

\end{cor}

 \begin{prop}\label{Kato_left_right}
Let $ T \in \mathcal{B}(X) $. Then, equivalent are:
\begin{enumerate}[label={\upshape(\roman*)}]
    \item $ T  $ is essentially left (right) Drazin invertible.
    \item There exist  a pair $(M,N) \in Red(T)$   such that  $T_{M} $ is  Kato and  left (right) Fredholm  and  $T_{N} $ is  nilpotent.
\end{enumerate}
   \end{prop}

\begin{proof}
(i)$\Longrightarrow$(ii)
According to Theorem \ref{Kato_essen}, there exists a pair $(M,N) \in Red(T)$  where $T_{M} $ is  left (right) Fredholm and $T_{N} $ is nilpotent. Since $T_M \in \mathcal{B}(M)$ is left (right) Fredholm, it follows from  \cite[Theorem 16.21]{Mull} and \cite[Lemma 2.1]{ZSR} that there exist closed subspaces $ M_1 $ and $ M_2$   of $M$   such that, $ \dim(M_2) < \infty$ and $(M_1, M_2)\in Red(T_M)$ where $T_{M_1} $ is  Kato and left (right) Fredholm and $T_{M_2} $ is nilpotent. Set $ G = M_1$ and $ H = M_2 \oplus N $ which is obviously closed. Then, clearly one has $(G, H)\in Red(T)$ where $T_{G} $ is  Kato and  left (right) Fredholm  and $T_{H} = T_{M_2} \oplus T_N$ is nilpotent.

(ii)$\Longrightarrow$(i) It follows from Theorem \ref{Kato_essen}.
\end{proof}

In what follows, we will denote the sets of Saphar type operators, left Drazin invertible operators, right Drazin invertible operators, essentially left Drazin invertible operators, and essentially right Drazin invertible operators respectively by \( S_i(X) \), where \( i \in \{1, 2, 3, 4, 5\} \). Additionally, the sets of Saphar type operators of degree \( \leq 1 \), left Drazin invertible operators with ascent \( \leq 1 \), right Drazin invertible operators with descent \( \leq 1 \), essentially left Drazin invertible operators with essential ascent \( \leq 1 \), and essentially right Drazin invertible operators with essential descent \( \leq 1 \) will be denoted respectively by \( S_i^1(X) \), where \( i \in \{1, 2, 3, 4, 5\} \).

\begin{thm}  Let $ T \in \mathcal{B}(X) $ and $ i \in \lbrace 1,2,3,4,5 \rbrace$. Then, the following  statements are equivalent:
\begin{enumerate}[label={\upshape(\roman*)}]
\item $T \in S_i(X) $.
\item There exist operators $ S, R \in \mathcal{B}(X) $ such that
 $ T = S + R $, $ SR = RS = 0 $,
 $ S \in S_{i}^1(X) $  and  $ R $ is nilpotent.

\item  There exist operators $ S, R \in \mathcal{B}(X) $ such that
 $ T = S + R $, $ SR = RS = 0 $,
 $ S \in S_i(X) $ and $ R $ is nilpotent.
\end{enumerate}
\end{thm}
\begin{proof}
(i)$\Longrightarrow$(ii) Let  $ T \in S_i(X) $, $i\in\{1,\dots,5\}$. Then there exists a pair    $(M, N)\in Red(T)$ such that  $T_M$ is Saphar and $ T_N $ is nilpotent, and  moreover in the case that  $i=2 $ (resp., $i=3$, $i=4$, $i=5$) $T_M$ is left invertible (resp., right invertible, left Fredhom, right Fredholm) according to Theorem \ref{left}  and Proposition \ref{Kato_left_right}.
  According to Lemma \ref{result1}  there exist two operators $S, R \in \mathcal{B}(X)$ such that $T= S+R$,  $SR =RS =0$,  $ R$ is nilpotent, $ S $ is of Saphar type and for all $ n \geq 1 $:
\begin{eqnarray}
   \mathcal{N}(S) \cap \mathcal{R}(S^n) &=& \mathcal{N}(T_M),\label{nov3}\\
   \mathcal{R}(S)+ \mathcal{N}(S^n) &=& \mathcal{R}(T_M) \oplus N.\label{nov4}
 \end{eqnarray}
If $ i=1$, obviously we have from (\ref{nov3}) that $ \dis(S) \leq 1 $. If $i=2$ (resp., i=4),  then 
from (\ref{nov3}) it follows  that $a(S)\le 1$ (resp.,   $a_e(S)\le 1$). If $i=3$ (resp., i=5
),  then
from (\ref{nov4}) it follows  that $d(S)\le 1$ (resp.,   $d_e(S)\le 1$).  Now applying  \cite[Corollary 4.23]{ZS} (resp., \cite[Corollary 4.24]{ZS}, \cite[Theorem 4.13]{ZS}, \cite[Theorem 4.15]{ZS}) we get  that $S\in S_i^1(X)$ for $i=2$ (resp., $i=3$, $i=4$, $i=5$).

(ii)$\Longrightarrow$(iii)  is trivial.

(iii)$\Longrightarrow$(i)  Suppose that $S\in S_i(X)$, $i\in\{1,\dots, 5\}$, $R$ is nilpotent, $SR=RS=0$ and $T=S+R$. Then S is of Saphar type.  From Corollary \ref{result3},   for all $n\ge r := \max\{d,\dis(S)\} $ (where $d$ is the degree of $ R$), one has that
$$ \mathcal{N}(T) \cap  \mathcal{R}(T^n) = \mathcal{N}(S) \cap  \mathcal{R}(S^{\dis(S)}),$$
Therefore, $T$ has  uniform descent for $n\ge r$ and $ \mathcal{N}(T) \cap  \mathcal{R}(T^r)$ is topologically complemented. Also, one has from Corollary \ref{result3} that    $$ \mathcal{R}(T)  + \mathcal{N}(T^r) = \mathcal{R}(S)  + \mathcal{N}(S^{\dis(S)}) $$ and hence $ \mathcal{R}(T)  + \mathcal{N}(T^r)$ is topologically complemented.
Applying \cite[Theorem 4.2]{ZS} we conclude that $T$ is of Saphar type. For $i=2$ (resp., $i=3, i=4, i=5$), as
$a(S)<\infty$ (resp., $ d(S)<\infty$, $a_e(S)<\infty$,  $ d_e(S)<\infty$),
 from Corollary \ref{result4},  it follows that $ a(T)<\infty$ (resp., $ d(T)<\infty$, $a_e(T)<\infty$, $ d_e(T)<\infty$). Now from \cite[Corollary 4.23]{ZS}  (resp., \cite[Corollary 4.24]{ZS}, \cite[Theorem 4.13]{ZS}, \cite[Theorem 4.15]{ZS}), we get  that $T\in S_i(X)$.
\end{proof}
\begin{rem}
    We shall note that the case $ i = 1$ in the last theorem complements  \cite[Theorem 5.2]{mine2}.
\end{rem}
\section{Iterates of (Essentially) Left and (Essentially) Right Drazin Invertible Operators}

The purpose of this section is to investigate the behavior of powers of (essentially) left and (essentially) right Drazin invertible operators.

 \begin{thm}\label{200} For $ T \in \mathcal{B}(X)$, the following statements are equivalent:
\begin{enumerate}[label={\upshape(\roman*)}]
\item $ T $ is (essentially) left Drazin invertible.
\item $ T^m $ is (essentially) left Drazin invertible for all $ m\in\NN $.
\item  $ T^m $ is (essentially) left Drazin invertible for some  $ m\in\NN $.
\end{enumerate}
\end{thm}
\begin{proof} (i)$\Longrightarrow$(ii):
If $ T $ is (essentially) left Drazin invertible, then by Theorem \ref{left} (\cite[Theorem 4.13]{ZS}),  there exist   closed $T$-invariant subspaces $ M $ and $ N $ of $X$   such that $ X = M \oplus N $  and  $ T = T_{M} \oplus T_{N} $ where $T_{M} $ is  left invertible (left Fredholm) and $T_{N} $ is nilpotent. For any  $ m \in\NN$, we have $ T^m = T_M^m \oplus T_N^m $. But $ T_M^m $ is left invertible (left Fredholm) and $ T_N^m $ is nilpotent. Thus, $ T^m $ is left Drazin invertible, again  by Theorem \ref{left} (\cite[Theorem 4.13]{ZS}).

(ii)$\Longrightarrow$(iii): It is obvious.

(iii)$\Longrightarrow$(i): Assume that $ T^m $ is left Drazin invertible for some  $ m \in\NN$.
 Let  $ d := a(T^m) $.
   Then  $ p := a(T) \leq md $. According to the equivalence (i)$\Longleftrightarrow$(iii) in Theorem \ref{t1}, to prove that $
T $ is left Drazin invertible it suffices to prove that  $ \mathcal{N}(T^{md}) $ and  $ \mathcal{R}(T^{md+1})$   are topologically complemented.
   From the fact that $ T^m $ is left Drazin invertible according to  the equivalence (i)$\Longleftrightarrow$(ii) in  Theorem \ref{t1} it follows that  $ \mathcal{N}((T^{m})^d)=\mathcal{N}(T^{md}) $ and $ \mathcal{R}((T^m)^{d+1})=\mathcal{R}(T^{md+m})$
 are topologically complemented.
 Notice that $ T^{m-1}(\mathcal{R}(T^{md+1})) = \mathcal{R}(T^{md+m}) $. Also,  since $ a(T) \leq md $, we know from \cite[Lemma 3.1]{kash}  that $ \mathcal{N}(T^{m-1}) \cap \mathcal{R}(T^{md+1}) = \lbrace 0 \rbrace $. As  $ \mathcal{R}(T^{md+1}) $ is paracomplete, an application of Lemma \ref{main l} gives that  $  \mathcal{R}(T^{md+1}) $ is topologically complemented.

 Suppose now that   $ T^m $ is essentially left Drazin invertible for some  $ m\in\NN $.
Then   $ d := a_e(T^m)<\infty $. From  the inequality  $\alpha_{mn}(T)\le \alpha_n(T^m)$, $n\in\NN$ \cite[p. 137]{Mull-Mb}, we have  that $a_e(T)\le ma_e(T^m)=md$.
On the other side,
 from  Theorem \ref{t1 e dod}  it follows that  $ \mathcal{N}((T^{m})^d)=\mathcal{N}(T^{md}) $ and $ \mathcal{R}((T^m)^{d+1})= \mathcal{R}(T^{md+m})$
 are topologically complemented. We have  that $$\dim  (\mathcal{N}(T^{m}) \cap \mathcal{R}(T^{md})) =\alpha_d(T^m)<\infty,$$ and hence $\dim  T(\mathcal{N}(T^{m}) \cap \mathcal{R}(T^{md})) <\infty$. Since $ T(\mathcal{N}(T^{m}) \cap \mathcal{R}(T^{md}))=\mathcal{N}(T^{m-1}) \cap \mathcal{R}(T^{md+1})$, we obtain that $\mathcal{N}(T^{m-1}) \cap \mathcal{R}(T^{md+1})$ is topologically  complemented.
   As in the previous part of the proof by using   Lemma \ref{main l}  we get  that  $  \mathcal{R}(T^{md+1}) $ is topologically complemented. Now from the equivalence (i)$\Longleftrightarrow$(x)  in \cite[Theorem 4.13]{ZS} we conclude that $
T $ is essentially left Drazin invertible.
\end{proof}

 \begin{thm}\label{100} For $ T \in \mathcal{B}(X)$, the following statements are equivalent:
\begin{enumerate}[label={\upshape(\roman*)}]
\item $ T $ is (essentially) right Drazin invertible.
\item $ T^m $ is (essentially) right Drazin invertible for all  $ m\in\NN $.
\item  $ T^m$ is (essentially) right Drazin invertible for some  $ m \in\NN $.
\end{enumerate}
\end{thm}
\begin{proof}
 (i)$\Longrightarrow$(ii): This can be proved   in a way similar  to the proof of the implication (i)$\Longrightarrow$(ii) in Theorem \ref{200}..

(ii)$\Longrightarrow$(iii): It is obvious.

(iii)$\Longrightarrow$(i): Let $ T^m $ be  right Drazin invertible for some  $ m \in\NN $.
 Let  $ d := d(T^m) $.
   Then we have that  $ d(T) \leq md $.
From the equivalence (i)$\Longleftrightarrow$(ii) in  Theorem \ref{t2} it follows that  $ \mathcal{N}((T^{m})^{d+1})=\mathcal{N}(T^{md+m}) $ and $ \mathcal{R}((T^m)^{d})= \mathcal{R}(T^{md})$
 are topologically complemented.
Since $ d(T) \leq md $,  from \cite[Lemma 3.2]{kash} it follows  that $ \mathcal{R}(T^{m-1}) + \mathcal{N}(T^{md+1}) = X $. As the subspace  $ \mathcal{N}(T^{md+1}) $ is   paracomplete and  $ T^{-(m-1)}(\mathcal{N}(T^{md+1})) = \mathcal{N}(T^{md+m}) $ is topologically complemented, an application of Lemma \ref{main r} gives that  $  \mathcal{N}(T^{md+1}) $ is topologically complemented.  According to  the equivalence (i)$\Longleftrightarrow$(iii) in Theorem \ref{t2} we conclude  that $
T $ is right Drazin invertible.

Now suppose that  $ T^m $ is essentially right Drazin invertible for some  $ m \in\NN $.
Then   $ d := d_e(T^m)<\infty $. From  \cite[Lemma 3]{Mull-Mb} it follows   that $d_e(T)\le md_e(T^m)=md$.
  By using the equi\-valence (i)$\Longleftrightarrow$(ii) in Theorem \ref{t2 e dod}  we obtain  that  $ \mathcal{N}((T^{m})^{d+1})=\mathcal{N}(T^{md+d}) $ and $ \mathcal{R}((T^m)^{d})= \mathcal{R}(T^{md})$
 are topologically complemented. Since $$\codim  (\mathcal{R}(T^{m}) +\mathcal{N}(T^{md})) =\beta_d(T^m)<\infty,$$ we have that $\codim\,  T^{-1}(\mathcal{R}(T^{m}) + \mathcal{N}(T^{md})) <\infty$. From \cite[Lemma 12]{Mull-Mb} it follows that $\mathcal{R}(T^{m}) + \mathcal{N}(T^{md})$ is closed and hence  $T^{-1}(\mathcal{R}(T^{m}) + \mathcal{N}(T^{md}))$ is closed. Since   $ T^{-1}(\mathcal{R}(T^{m}) + \mathcal{N}(T^{md}))=\mathcal{R}(T^{m-1}) + \mathcal{N}(T^{md+1})$, we get  that $\mathcal{R}(T^{m-1})+ \mathcal{N}(T^{md+1})$ is topologically  complemented.
   As in the previous part of the proof, Lemma \ref{main r}  ensures that  $  \mathcal{N}(T^{md+1}) $ is topologically complemented.
   Now  from the equivalence (i)$\Longleftrightarrow$(iii) in  Theorem  \ref{t2 e dod}, it follows that $
T $ is essentially right Drazin invertible.
\end{proof}

\paragraph{Author contributions}: All authors contributed to the study conception and
design. Also, all authors read and approved the final manuscript.

\paragraph{Competing Interests}: The authors declare that they have no financial interests.
\paragraph{Funding}: Sne\v zana \v C. \v Zivkovi\'c-Zlatanovi\'c is supported by
the Ministry of Education, Science and Technological Development, Republic
of Serbia, grant no. 451-03-68/2022-14/200124.

\paragraph{Data availability statement}: Data sharing is not applicable to this article as no new data were created or analyzed in this study.
\paragraph{Statements and Declarations}: Author have no conflict of interest to declare.

\end{document}